\documentclass{amsart}
\usepackage{amsthm,amssymb,amsmath,hyperref}
\usepackage{graphicx}
\usepackage{color}
\usepackage[table]{xcolor}
\usepackage{tikz}
\usetikzlibrary{decorations.pathreplacing}
\usepackage[utf8]{inputenc}
\usepackage{textgreek}
\usepackage{mathtools}

\usepackage{multirow}

\theoremstyle{plain}
\newtheorem{thm}{Theorem}[section]

\newtheorem{lem}[thm]{Lemma}
\newtheorem{prop}[thm]{Proposition}

\theoremstyle{definition}
\newtheorem{defn}[thm]{Definition}

\theoremstyle{remark}
\newtheorem{rem}[thm]{Remark}
\newtheorem{exa}[thm]{Example}
\theoremstyle{plain}

\numberwithin{equation}{section}

\newcommand{\del}{\delta}

\newcommand{\R}{{\mathbb R}}

\newcommand{\Q}{{\mathbb Q}}
\newcommand{\N}{{\mathbb N}}

\newcommand{\Z}{{\mathbb Z}}

\newcommand{\calD}{{\mathcal D}}

\newcommand{\calJ}{{\mathcal J}}

\def\udot#1{\ifmmode\oalign{$#1$\crcr\hidewidth.\hidewidth
    }\else\oalign{#1\crcr\hidewidth.\hidewidth}\fi}

\def\R{\mathbb{R}}
\def\Z{\mathbb{Z}}

\def\Q{\mathbb{Q}}

\begin{document}
	
\title[]{A structure theorem on doubling measures with different bases}
\author{Theresa C. Anderson and Bingyang Hu}

\address{Theresa C. Anderson: Department of Mathematics, Purdue University, 150 N. University St., W. Lafayette, IN 47907, U.S.A.}%
\email{tcanderson@purdue.edu}

\address{Bingyang Hu: Department of Mathematics, Purdue University, 150 N. University St., W. Lafayette, IN 47907, U.S.A.}%
\email{hu776@purdue.edu}

\begin{abstract}
In this paper, we prove a structure theorem for the infinite union of $n$-adic doubling measures via techniques which involve far numbers. Our approach extends the results of Wu in 1998, and as a by product, we also prove a classification result related to normal numbers. 
\end{abstract}
\date{\today}

\thanks{The first author is funded by NSF DMS 1954407 in Analysis and Number Theory.}

\maketitle

\section{Introduction}
The goal of this paper is to establish a structure theorem for $n$-adic doubling measures (including dyadic doubling measures $n=2$), a key area of importance in harmonic analysis and related fields, as discussed below.  As a byproduct of our work, we prove a statement involving a classification related to \emph{normal numbers}, a popular topic in number theory, that has nothing to do with measures.

We begin with recalling several definitions. A \emph{doubling measure} $\mu$ is a measure for which there exists a positive constant $C$ such that for every interval $I \subset \R, \mu(2I) \le C\mu(I)$, where $2I$ is the interval which shares the same midpoint of $I$ and twice the length of $I$.

We will focus on $n$-adic intervals, $ n \in \N$: 

$$
I=\left[ \frac{k-1}{n^\ell}, \frac{k}{n^{\ell}} \right), \quad \ell, k \in \Z. 
$$ 
The $n$-adic children of the interval defined above are 
\begin{equation} \label{20201119eq01}
I_j=\left[\frac{k-1}{n^{\ell}}+\frac{j-1}{n^{\ell+1}}, \frac{k-1}{n^{\ell}}+\frac{j}{n^{\ell+1}} \right), \quad 1 \le j \le n. 
\end{equation} 
A measure $\mu$ is a \emph{$n$-adic doubling measure on $\R$} if there exists a positive constant $C$, independent of all parameters, such that for any $n$-adic $I$,
\begin{equation} \label{20201121eq11}
\frac{1}{C} \le \frac{\mu(I_{j_1})}{\mu(I_{j_2})} \le C, \quad 1 \le j_1, j_2 \le n,
\end{equation}
where both $I_{j_1}$ and $I_{j_2}$ are some $n$-adic children of $I$ defined in \eqref{20201119eq01}. We denote the collection of all $n$-adic doubling measures by $\calD_n$. 

Doubling measures are a classical topic in analysis and they have many deep connections to other fields, such as PDE.   The well-known Muckenhoupt $A_p$ weights and reverse H\"older weights are all automatically doubling, and the doubling property is the key point of the definition of spaces of homogeneous type.  For more background and applications of doubling measures (particularly from a more modern perspective), see, for example \cite{Conde}, \cite{DCU}, \cite{HK}, \cite{LN}, \cite{LPW}, \cite{TM}, \cite{P}, \cite{PWX}. In particular, dyadic doubling measures have been central to a rich area of study, see, for example \cite{FKP}.  The study of the union as well as the intersection of $n$-adic doubling measures is a recent topic. It dates back to Wu's work \cite{Wu1, Wu2} in 90s on using the null set to characterize the $n$-adic doubling measures, in particular, she proved the following result. 

\begin{thm} [\cite{Wu1}] \label{Wuthm}
For any two integers $A$ and $B$ greater than $2$, $\frac{\log A}{\log B}$ is irrational if and only if $\calD_A \not\subseteq \calD_B$ and $\calD_B \not\subseteq \calD_A$. 
\end{thm}

A natural question to ask is whether we can extend the above result to intersection or unions of $n$-adic measures. Here is some recent work along this line of research. 

\begin{enumerate}
    \item [1.] In unpublished work of Jones, he proved that any finite intersection of the \emph{prime $BMO$ function classes} is never equal to the full $BMO$ function class (for more details about this see \cite{Krantz}).  Since then, a folkloric question has been: ``does an analogue of Jones's result hold for $p$-adic doubling measures?"  This proved a difficult extension that was answered recently by \cite{BMW} and \cite{AH2}, described in items 2. and 3. below.
    
    
    \medskip
    
    \item [2.]  The first work which extends Wu's type of results to the union or the intersections of $D_n$'s was due to Boylan, Mills and Ward \cite{BMW}, which also was the first step to answering the folkloric analogue of Jones's question above;
    
    \medskip
    
    \item [3.] In our recent work \cite{AH2}, we answer the analogue of Jones's question for measures by proving that for any finite family of primes $p_i$, there exists a measure that is $p_i$-adic doubling yet not doubling.  Additionally we extend this result to the setting of Muckenhoupt $A_p$ and reverse H\"older weights.
 \end{enumerate}
 
 \medskip
 
 The results in \cite{AH2} left several open question (both implicit and explicitly stated).  In this paper, we completely resolve one of these by proving the following structure theorem for unions of $n$-adic measures. Here is the main result.

\begin{thm} \label{mainthm}
Let $\{n_i\}$ and $\{m_j\}$ be any sequences of integers greater than $2$. Then the following statements are equivalent:
\begin{enumerate}
\item [(1).] There exists some $i \ge 1$ such that
$$
\frac{\log n_i}{\log m_j} \in \R \backslash \Q, \quad \forall j \ge 1; 
$$
\item [(2).]
$$
 \bigcup\limits_{i \ge 1} \calD_{n_i}  \not\subseteq \bigcup\limits_{j \ge 1} \calD_{m_j}.
$$
\end{enumerate}
\end{thm}

\begin{rem}
Here are some remarks for Theorem \ref{mainthm}.
\begin{enumerate}
    \item [1.] Clearly, Theorem \ref{mainthm} generalizes Theorem \ref{Wuthm}; 
    
    \medskip
    
    \item [2.] 
 The proof of Theorem \ref{Wuthm} by Wu is based on Kronecker's theorem on irrational numbers, namely, for $r $ being irrational, the set $\{kr \ (\textrm{mod} \ 1): k \in \Z\}$ is dense on $[0, 1)$. To our best knowledge, Wu's approach seems difficult to extend to the situation when we try to understand the behavior of the union or the intersection of $\calD_n$. More precisely, for $x \in \R$ with $\|x\|>\frac{1}{10}$, where $\|x\|$ denotes the distance form $x$ to the nearest integer, let
    $$
    \calJ:=\left\{j: j \in \N, \ \|jx\|> \frac{1}{20} \right\}
    $$
     The proof in \cite{Wu1} relies heavily on defining corresponding $\calJ$ for a specially chosen sequence of $x_i\in \R$ and noting that for all these $x_i$ (where $\|x_i\| >\frac{1}{10}$), $\calJ^c$ contains no consecutive integers.  To establish a result as in Theorem \ref{mainthm}, we likely would need to do the same for the sets
    $$
    \widetilde{\calJ}:=\left\{j: j \in \N, \ \|jx_{1, i}\|, \ \|jx_{2, i}\|>\frac{1}{20} \right\}
    $$
    where $\|x_{1, i}\|, \|x_{2, i}\|>\frac{1}{10}$, which would in turn involve showing that unions of sets like $\calJ^c$ do not contain consecutive integers.  This appears quite challenging.  The rest of Wu's approach relies on an intricate construction that uses several highly technical lemmas that she had developed.  This allowed for a precise, explicit construction, but made direct generalizations of her techniques even less amenable.  Hence we decided to take a completely different approach;

\medskip

    \item [3.]  Our approach is different from Wu and based a systematic study of \emph{far numbers} on $\R$ (see, \cite{AHJOW}). The advantage of our method is two-fold. First, our approach is simpler and less technically reliant than the argument in \cite{Wu1}, avoiding the construction of certain auxiliary $n$-adic doubling measures and iteration arguments. Second, we are able to  
    extend this type of structure theorem to the union of $n$-adic measures (namely, Theorem \ref{mainthm}); 
    
\medskip

    \item [4.]  Note that $\calD_m = \calD_{m^a}$ for any integer $a \geq 1$.  This fact follows easily from the definition, but will contribute to an important reduction step in our proof (see Section \ref{20201207sec01}).  Note that this fact can be thought of as a version of \emph{Hensel's lemma} in analysis: that is, that $m$-adic doubling measures (and in particular $p$-adic doubling measures for any prime $p$) can be ``lifted" to $m^a$-adic doubling measures for $a \geq 1$.     
\end{enumerate}
\end{rem}

We prove this via an explicit construction that heavily weaves in number theory.

\begin{rem}
An \emph{important} observation for the construction (see the next section for details) is that all the intervals involved only depend on $\{n_i\}$ and are independent of $\{m_j\}$. This is very different from
\begin{enumerate}
\item [1.] The null set construction in \cite{Wu1}; 

\medskip

\item [2.] The exotic measures in \cite{AH2, BMW}.  
\end{enumerate}
This is why we are allowed to prove a structure theorem for infinite unions of $n$-adic measures instead of simply finite ones.. 
\end{rem}

To begin with, we recall the definition of \emph{$n$-far numbers}. 

\begin{defn} \label{defn01}
A real number $\del$ is \emph{$n$-far} if the distance from $\del$ to each given rational $k/n^m$ is at least some fixed multiple of $1/n^m$, where $m \ge 0$ and $k \in \Z$. That is, if there exists $C>0$ such that
\begin{equation}
\label{C delta}
   \left| \del-\frac{k}{n^m} \right| \ge \frac{C}{n^m}, \quad \forall m  \ge 0, \ k \in \Z, 
\end{equation}
where $C$ may depend on $\del$ but independent of $m$ and $k$. Note also that $0<C<1$. 
\end{defn}

\begin{rem}
Note that in the first condition in Theorem \ref{normalcla}, we cannot interchange the role of $m$ and $n$. Indeed, for example, $\frac{1}{6}$ is $2$-far while $\frac{1}{2}$ is not $6$-far. 
\end{rem}

We need the following result from \cite{AHJOW}.

\begin{prop} \label{20201120prop01}
All rationals except those of the form $\frac{k}{n^m}, m \ge 0, k \in \Z$ are $n$-far numbers. 
\end{prop}

The rest of the paper is devoted to prove Theorem \ref{mainthm}.

\bigskip
\section{Proof of the main result: Part I.}
We begin with observing that it is suffices to check that (2) implies (1). Indeed, assuming (1) fails, we see that for each $i \ge 1$, there exists some $j_i \ge 2$, such that $\frac{\log n_i}{\log m_{j_i}} \in \Q$, which means there exists some $\mathfrak n_i \in \Z$ such that both $n_i$ and $m_{j_i}$ are some powers of $\mathfrak n_i$. This further implies for each $i \ge 1$, $\calD_{\mathfrak n_i}= \calD_{n_i}=\calD_{m_{j_i}}$ (this fact is easy to see using the definition), which contradicts (2). 

In the rest of this paper, we prove (1) implies (2). The desired result clearly follows from the following result.

\begin{thm} \label{20201121thm01}
Let $n:=n_i$ and $\{m_j\} \subseteq \Z$ satisfy the assumption in (1). Then there exists a measure $\mu$ on $\R$, such that $\mu$ is $n$-adic doubling but not $m_j$-adic doubling for any $j \ge 1$. 
\end{thm}

\subsection{Construction of $\mu$.} \label{20201202sub01} To begin with, we simply let the restriction of $\mu$ on $(-\infty, 0)$ be the Lebesgue measure, and we will re-distribute the weights on $[0, \infty)$.  Take any $a, b>0$ with 
\begin{equation} \label{20201120eq01}
a+b=2 \quad \textrm{and} \quad 0<a<1<b. 
\end{equation}
Now on each $\ell \in \N$, we re-distribute the weight on the interval $[\ell, \ell+1)$ as follows.

\medskip

\textit{Step I:} Note that $[\ell, \ell+1)$ is a $n$-adic interval. We assign the weight $a$ to its leftmost $n$-adic child and $b$ to its rightmost $n$-adic child, that is
$$
\mu \big |_{\left[ \ell, \ell+\frac{1}{n} \right)}=a dx \quad \textrm{and} \quad \mu \big |_{\left[\ell+\frac{n-1}{n}, \ell+1 \right)}=b dx. 
$$
While for all other $n$-adic children, the weights there remain unchanged; 

\medskip

\textit{Step II:} Repeat the procedure $\ell+1$ times in Step I to all the $n$-adic children whose weights have been redistributed from the previous step. For example, for the $n$-adic child $\left[\ell, \ell+\frac{1}{n} \right)$ which has been selected in Step I, we let
$$
d\mu \big |_{\left[ \ell, \ell+\frac{1}{n^2} \right)}=(a^2)dx \quad \textrm{and} \quad d\mu \big |_{\left[\ell+\frac{n-1}{n^2}, \ell+\frac{1}{n} \right)}=(ab) dx, 
$$
and
$$
d\mu \big |_{\left[\ell+\frac{j}{n^2}, \ell+\frac{j+1}{n^2} \right)}=a dx, \quad j \in \{1, \dots, n-2\}. 
$$
\medskip

We plot the measure $\mu$ when $n=3$ (see, Figure \ref{20201121Fig01}). 

\medskip

\begin{figure}[ht]
\begin{tikzpicture}[scale=5.5]
\draw (-.15,0) -- (2.1,0); 
\fill (0,0) circle [radius=.2pt];
\fill (1,0) circle [radius=.2pt];
\fill (2, 0) circle [radius=.2pt]; 
\fill (1/3, 0) circle [radius=.2pt];
\fill (2/3, 0) circle [radius=.2pt];
\fill (10/9, 0) circle [radius=.2pt];
\fill (11/9, 0) circle [radius=.2pt];
\fill (4/3, 0) circle [radius=.2pt];
\fill (5/3, 0) circle [radius=.2pt];
\fill (16/9, 0) circle [radius=.2pt];
\fill (17/9, 0) circle [radius=.2pt];
\fill (0, -0.01) node [below] {{\footnotesize $0$}};
\fill (1, -0.01) node [below] {{\footnotesize $1$}};
\fill (2, -0.01) node [below] {{\footnotesize $2$}};
\fill (1/3, -0.01) node [below] {{\footnotesize $\frac{1}{3}$}};
\fill (2/3, -0.01) node [below] {{\footnotesize $\frac{2}{3}$}};
\fill (10/9, -0.01) node [below] {{\footnotesize $\frac{10}{9}$}};
\fill (11/9, -0.01) node [below] {{\footnotesize $\frac{11}{9}$}};
\fill (4/3, -0.01) node [below] {{\footnotesize $\frac{4}{3}$}};
\fill (5/3, -0.01) node [below] {{\footnotesize $\frac{5}{3}$}};
\fill (16/9, -0.01) node [below] {{\footnotesize $\frac{16}{9}$}};
\fill (17/9, -0.01) node [below] {{\footnotesize $\frac{17}{9}$}};
\fill (-.075, .01) node [above] {{\footnotesize {\color{blue} $1$}}}; 
\fill (1/6, .01) node [above] {{\footnotesize {\color{blue} $a$}}};
\fill (1/2, .01) node [above] {{\footnotesize {\color{blue} $1$}}}; 
\fill (5/6, .01) node [above] {{\footnotesize {\color{blue} $b$}}};
\fill (19/18, .01) node [above] {{\footnotesize {\color{blue} $a^2$}}}; 
\fill (7/6, .01) node [above] {{\footnotesize {\color{blue} $a$}}};
\fill  (23/18, .01) node [above] {{\footnotesize {\color{blue} $ab$}}}; 
\fill (1.5, .01) node [above] {{\footnotesize {\color{blue} $1$}}}; 
\fill (31/18, .01) node [above] {{\footnotesize {\color{blue} $ba$}}};
\fill (33/18, .01) node [above] {{\footnotesize {\color{blue} $b$}}};
\fill (35/18, .01) node [above] {{\footnotesize {\color{blue} $b^2$}}}; 
\fill (2.05, .01) node [above] {{\footnotesize {\color{blue} $a^3$}}};
\end{tikzpicture}
\caption{$\mu$ with $n=3$, where the blue parts refer to the weights associated to each interval.}
\label{20201121Fig01}
\end{figure}
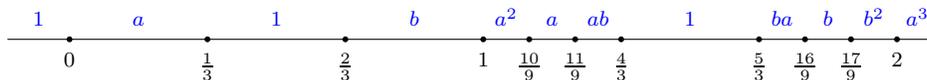

\medskip

We have the following observation.

\begin{lem}
$\mu$ is $n$-adic doubling but not doubling.
\end{lem}

\begin{proof}
It is clear that  $\mu$ is $n$-adic doubling, as the ratio \eqref{20201121eq11} for any $n$-adic interval is either $\frac{a}{b}, 1$ or $\frac{b}{a}$. 

Next we argue that $\mu$ is not doubling. Indeed, this follows by comparing the weights near the integer points. More precisely, for each $\ell \ge 1$, note that on its left hand side, there is an $n$-adic interval with weight $b^\ell$ with sidelength $\frac{1}{n^\ell}$; while on its right hand side, there is an $n$-adic interval with weight $a^{\ell+1}$ with sidelength $\frac{1}{n^{\ell+1}}$. Then, for example, we might consider the interval
$$
\left[\ell-\frac{1}{n^{\ell+1}}, \ell+\frac{1}{n^{\ell+1}} \right), 
$$
whose left half has weight $b^\ell$, while right half has weight $a^{\ell+1}$. The desired claim is clear. 
\end{proof}

\medskip

In the rest of the paper, we show that the measure $\mu$ constructed above is not $m$-adic doubling with $m$ satisfying 
\begin{equation} \label{20201123eq01}
\frac{\log n}{\log m} \in \R \backslash \Q. 
\end{equation} 
Without loss of generality, we may assume $n \ge 3$, otherwise, that is if $n=2$, we simply replace it by $n=4$, as $\calD_2=\calD_4$. 

\medskip

We include some motivation before we proceed. The \emph{key} observation is that condition \eqref{20201123eq01} is indeed \emph{not} equivalent to the far number characterization. For example, $\frac{\log 12}{\log 18}$ is irrational, however, neither $\frac{1}{12}$ is $18$-far nor $\frac{1}{18}$ is $12$-far, since $\frac{1}{12}=\frac{27}{18^2}$ and $\frac{1}{18}=\frac{8}{12^2}$, respectively. This suggests us to consider two different cases:

\begin{enumerate}
    \item [I.] $\frac{1}{n}$ is $m$-far;
    
    \medskip
    
    \item [II.] $\frac{1}{n}$ is not $m$-far. 
\end{enumerate}

For the first case, since $\frac{1}{n}$ is $m$-far, this means $\frac{1}{n}$ is ``far away" from all rationals of the form $\frac{k}{m^\ell}$, and hence it suffices to consider sufficient small $m$-adic intervals near the points $\ell+\frac{1}{n}$ for $\ell$ large. While for the second case, we shall later see that this reduces to the solubility of a certain equation (see, \eqref{20201129eq01}), which allows us to modify the argument from the first case so that the ratio \eqref{20201119eq01} will ``blow up" as desired. 

In the rest of this section, we will focus on the case when $\frac{1}{n}$ is $m$-far, while we postpone the more complex second case to the next section. 

\subsection{Proof of Theorem \ref{20201121thm01}: the far number case}  For $\ell$ sufficiently large, we consider the point $\ell+\frac{1}{n}$ and the interval
$$
I_\ell:=\left[\ell+\frac{1}{n}-\frac{1}{n^{\ell+1}},  \ell+\frac{1}{n}+\frac{1}{n^{\ell+1}} \right).
$$
By the construction of $\mu$, the weights associated to $I_\ell$ are $ab^\ell$ on its left half and $1$ on its right half. 

Take any intetger $\ell'>\frac{(\ell+1) \log n}{\log m}$ and let $J_{\ell'}$ be the unique $m$-adic interval which contains $\ell+\frac{1}{n}$ with sidelength $m^{-\ell'}$. Note that by the choice of $\ell'$, $J_{\ell'} \subset I_\ell$ and moreover, since $\frac{1}{n}$ is $m$-far, $\ell+\frac{1}{n}$ does not coincide with any endpoints of $J_{\ell'}$. We consider three cases. To this end, we denote
$$
J_{\ell', L}:=\textrm{the leftmost $m$-adic child of $J_{\ell'}$}
$$
and
$$
J_{\ell', R}:=\textrm{the rightmost $m$-adic child of $J_{\ell'}$}.
$$

\medskip

\textit{Case I: $\ell+\frac{1}{n} \notin J_{\ell', L} \cup J_{\ell', R}$.} In this case, it suffices to note that
\begin{equation} \label{20201130eq01}
\frac{\mu(J_{\ell', L})}{\mu(J_{\ell', R})}=ab^\ell. 
\end{equation}
(see, Figure \ref{20201121Fig02}). 

\begin{figure}[ht]
\begin{tikzpicture}[scale=6]
\draw (-.15,0) -- (1.1,0); 
\fill (0,0) circle [radius=.2pt];
\fill (1,0) circle [radius=.2pt];
\fill (.5, 0) circle [radius=.2pt]; 
\draw [line width=0.8mm, green ] (.2, 0) -- (.3, 0);  
\draw [line width=0.8mm, cyan] (.65, 0) --(.75, 0); 
\fill (0, -.01) node [below] {{\footnotesize $\ell+\frac{1}{n}-\frac{1}{n^{\ell+1}}$}};
\fill (0.5, -.01) node [below] {{\footnotesize $\ell+\frac{1}{n}$}};
\fill (1, -.01) node [below] {{\footnotesize $\ell+\frac{1}{n}+\frac{1}{n^{\ell+1}}$}};
\draw [decorate,decoration={brace,amplitude=8pt,raise=10pt},yshift=2pt] (0, 0) -- (0.5,  0) node [black,midway,xshift=0cm, yshift=.9cm] {\tiny {{\color{blue} $ab^\ell$}}};
\draw [decorate,decoration={brace,amplitude=8pt,raise=10pt},yshift=2pt] (0.5, 0) -- (1,  0) node [black,midway,xshift=0cm, yshift=.9cm] {\tiny {{\color{blue} $1$}}};
\fill (0.2, 0) circle [radius=.15pt];
\fill (0.3, 0) circle [radius=.15pt];
\fill (0.65, 0) circle [radius=.15pt];
\fill (0.75, 0) circle [radius=.15pt]; 
\draw [decorate,decoration={brace,amplitude=8pt,raise=10pt},yshift=-2pt] (0.75, 0) -- (.2,  0) node [black,midway,xshift=0cm, yshift=-.9cm] {\tiny {{\color{red} $J_{\ell'}$}}};
\draw (.25, .01) node [above] {{\tiny {\color{red} $J_{\ell', L}$}}};
\draw (.7, .01) node [above] {{\tiny {\color{red} $J_{\ell', R}$}}};
\end{tikzpicture}
\caption{The far number case: Case I.}
\label{20201121Fig02}
\end{figure}
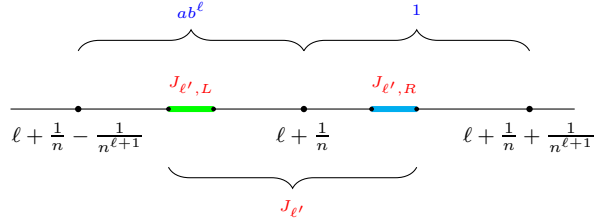

\medskip

\textit{Case II:  $\ell+\frac{1}{n} \in J_{\ell', L}$}. Let $l\left(J_{\ell'} \right)$ be the left endpoint for $J_{\ell'}$. Note that $l\left(J_{\ell'} \right)$ is of the form $\frac{k}{m^{\ell'}}$ for some $k \in \Z$, this implies that
$$
\left|l\left(J_{\ell'} \right)-\left(\ell+\frac{1}{n} \right) \right| \ge \frac{C}{m^{\ell'}},
$$
for some $C>0$ only depends on $m$ and $n$, where we have used the fact that $\frac{1}{n}$ is $m$-far. Moreover, we can also conclude that $C<\frac{1}{m}$. Indeed, since
$\ell+\frac{1}{n} \in J_{\ell', L}$ and $l\left(J_{\ell'} \right)+\frac{1}{m^{\ell'}} = \frac{k+1}{m^{\ell'}} \neq \ell+\frac{1}{n}$ (since $\frac{1}{n}$ is $m$-far), it follows that 
$$
\left|l\left(J_{\ell'} \right)-\left(\ell+\frac{1}{n} \right) \right|< \frac{1}{m^{\ell'+1}}, 
$$
which together with \eqref{20201130eq01} gives the desired assertion.

Therefore, we have
\begin{equation} \label{20201130eq02}
\frac{\mu\left(J_{\ell', L} \right)}{\mu \left(J_{\ell', R} \right)} \ge \frac{ab^\ell \cdot \frac{C}{m^{\ell'}}+\left(\frac{1}{m}-C \right) \cdot \frac{1}{m^{\ell'}}}{\frac{1}{m^{\ell'+1}}} \ge Cm \cdot ab^\ell. 
\end{equation} 
(see, Figure \ref{20201121Fig03}). 

\medskip

\begin{figure}[ht]
\begin{tikzpicture}[scale=6]
\draw (-.15,0) -- (1.1,0); 
\fill (0,0) circle [radius=.2pt];
\fill (1,0) circle [radius=.2pt];
\draw [line width=0.8mm, green ] (.43, 0) -- (.53, 0);  
\draw [line width=0.8mm, cyan] (.75, 0) --(.85, 0); 
\fill (.5, 0) circle [radius=.2pt]; 
\fill (0, -.01) node [below] {{\footnotesize $\ell+\frac{1}{n}-\frac{1}{n^{\ell+1}}$}};
\fill (0.5, -.01) node [below] {{\footnotesize $\ell+\frac{1}{n}$}};
\fill (1.02, -.01) node [below] {{\footnotesize $\ell+\frac{1}{n}+\frac{1}{n^{\ell+1}}$}};
\draw [decorate,decoration={brace,amplitude=8pt,raise=10pt},yshift=2pt] (0, 0) -- (0.5,  0) node [black,midway,xshift=0cm, yshift=.9cm] {\tiny {{\color{blue} $ab^\ell$}}};
\draw [decorate,decoration={brace,amplitude=8pt,raise=10pt},yshift=2pt] (0.5, 0) -- (1,  0) node [black,midway,xshift=0cm, yshift=.9cm] {\tiny {{\color{blue} $1$}}};
\fill (0.43, 0) circle [radius=.15pt];
\fill (0.53, 0) circle [radius=.15pt];
\fill (0.75, 0) circle [radius=.15pt];
\fill (0.85, 0) circle [radius=.15pt]; 
\draw [decorate,decoration={brace,amplitude=8pt,raise=10pt},yshift=-2pt] (0.85, 0) -- (.43,  0) node [black,midway,xshift=0cm, yshift=-.9cm] {\tiny {{\color{red} $J_{\ell'}$}}};
\draw (.48, .01) node [above] {{\tiny {\color{red} $J_{\ell', L}$}}};
\draw (.8, .01) node [above] {{\tiny {\color{red} $J_{\ell', R}$}}};
\end{tikzpicture}
\caption{The far number case: Case II.}
\label{20201121Fig03}
\end{figure}
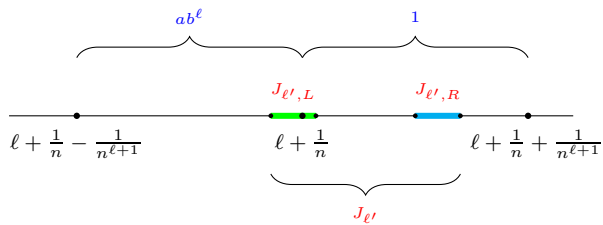

\medskip

\textit{Case III: $\ell+\frac{1}{n} \in J_{\ell', R}$.}  The third case can be treated as an application of the previous two cases. Indeed, by a similar argument as in Case II, we can see the ratio \eqref{20201119eq01} with respect to any two of the $m$-adic children of $J_{\ell'}$ is of size $1$. Therefore, instead of considering $J_{\ell'}$, we consider $J_{\ell', R}$ and check whether the point $\ell+\frac{1}{n}$ locates in the rightmost $m$-adic child of $J_{\ell', R}$ or not: if not, then we can apply the argument in Case I and Case II at this smaller scale to get estimates as in \eqref{20201130eq01} and \eqref{20201130eq02}; otherwise, we simply go down to the next smallest scale and repeat this procedure. Note that this procedure will stop in finite steps as the distance between $\ell+\frac{1}{n}$ and the right endpoint of $J_{\ell'}$ is fixed (see, Figure \ref{20201121Fig04}).

\medskip

\begin{figure}[ht]
\begin{tikzpicture}[scale=6]
\draw (-.15,0) -- (1.1,0); 
\fill (0,0) circle [radius=.2pt];
\fill (1,0) circle [radius=.2pt];
\draw [line width=0.8mm, green ] (.13, 0) -- (.23, 0);  
\draw [line width=0.8mm, cyan] (.47, 0) --(.57, 0); 
\fill (.5, 0) circle [radius=.2pt]; 
\fill (-.05, -.01) node [below] {{\footnotesize $\ell+\frac{1}{n}-\frac{1}{n^{\ell+1}}$}};
\fill (0.5, -.01) node [below] {{\footnotesize $\ell+\frac{1}{n}$}};
\fill (1, -.01) node [below] {{\footnotesize $\ell+\frac{1}{n}+\frac{1}{n^{\ell+1}}$}};
\draw [decorate,decoration={brace,amplitude=8pt,raise=10pt},yshift=2pt] (0, 0) -- (0.5,  0) node [black,midway,xshift=0cm, yshift=.9cm] {\tiny {{\color{blue} $ab^\ell$}}};
\draw [decorate,decoration={brace,amplitude=8pt,raise=10pt},yshift=2pt] (0.5, 0) -- (1,  0) node [black,midway,xshift=0cm, yshift=.9cm] {\tiny {{\color{blue} $1$}}};
\fill (0.13, 0) circle [radius=.15pt];
\fill (0.23, 0) circle [radius=.15pt];
\fill (0.47, 0) circle [radius=.15pt];
\fill (0.57, 0) circle [radius=.15pt]; 
\draw [decorate,decoration={brace,amplitude=8pt,raise=10pt},yshift=-2pt] (0.57, 0) -- (.13,  0) node [black,midway,xshift=0cm, yshift=-.9cm] {\tiny {{\color{red} $J_{\ell'}$}}};
\draw (.18, .01) node [above] {{\tiny {\color{red} $J_{\ell', L}$}}};
\draw (.52, .01) node [above] {{\tiny {\color{red} $J_{\ell', R}$}}};
\end{tikzpicture}
\caption{The far number case: Case III.}
\label{20201121Fig04}
\end{figure}
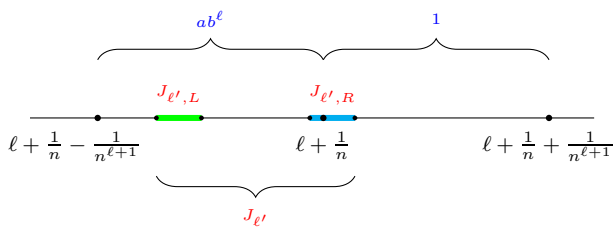

\medskip

In conclusion, if $\frac{1}{n}$ is $m$-far, we have proved that, for each interval $[\ell, \ell+1)$, there exists a sufficiently small $m$-adic interval containing the point $\ell+\frac{1}{n}$, such that the maximum of the ratio \eqref{20201119eq01} is bounded below by $\min \{Cm, 1\} \cdot ab^\ell$, which clearly blows up when $\ell$ converges to infinity.

\bigskip

\section{Proof of the main result: Part II.}

In this section, we prove Theorem \ref{20201121thm01} under the assumption that $\frac{1}{n}$ is \emph{not} $m$-far. To begin with, we first make a remark that our approach in the previous section might not work: indeed, since $\frac{1}{n}$ is not $m$-far, by Proposition \ref{20201120prop01}, for each $\ell \in \N$, we have
$$
\ell+\frac{1}{n}=\frac{k_\ell}{m^{s_\ell}}, \quad \textrm{for some} \quad k_\ell, s_\ell \in \N. 
$$
This implies $\ell+\frac{1}{n}$ can indeed be one of the endpoints when we restrict our attention to those small $m$-adic interval containing $\ell+\frac{1}{n}$ and hence we are not able to benefit any more from the fact that the weights on both sides of $\ell+\frac{1}{n}$ differ dramatically (see, e.g., Figure \ref{20201121Fig02}). Therefore, we have to refine the choice of $m$-adic interval for Case II above, which will entail a more concerted effort.

We shall see later that the desired refinement reduces to study the solubility of the equation
\begin{equation} \label{20201129eq01}
\frac{k}{m^{\ell-1}}=\frac{1}{n^\ell}, 
\end{equation} 
where $m, n \ge 2$. More precisely, if \eqref{20201129eq01} is \emph{unsolvable}, then we can refine our argument in the first case by considering sufficiently small $m$-adic intervals near points $\left\lfloor \frac{l \log m}{\log n} \right\rfloor-1+\frac{1}{n^\ell}$ for $\ell$ sufficiently large. 

For this purpose, we have the following definition.

\begin{defn}
Let $m, n$ be two integers greater than $2$. We say a pair $(m, n)$ is \emph{solvable} if there exists integers $k, \ell \ge 1$ such that \eqref{20201129eq01} holds. Otherwise, we say $(m, n)$ is \emph{unlovable}. 
\end{defn}

The plan of this section is as follows. In the first part, we study the solubility of the equation \eqref{20201129eq01}, more precisely, we shall show that one can ``transform" every solvable pair $(m, n)$ into an unsolvable pair while this ``transformation" is invariant under Theorem \ref{20201121thm01}. In the second part, we use the insolubility of the equation \eqref{20201129eq01} to complete the proof of Theorem \ref{20201121thm01}. 

\medskip

\subsection{Good pairs and semi-good pairs.} We first understand the solubility of the equation \eqref{20201129eq01}, with the assumption that $\frac{1}{n}$ is not $m$-far. 

Recall that by Proposition \ref{20201120prop01}, there exists some $k_0, s_0 \in \N$, such that 
$$
\frac{1}{n}=\frac{k_0}{m^{s_0}}, 
$$
that is, $k_0 n=m^{s_0}$. This implies that if $p$ is a prime factor of $n$, then so is $m$, and hence we can write the prime decomposition of $n$ and $m$ as follows:
\begin{equation} \label{20201201eq01}
n=p_1^{a_1} \dots p_N^{a_N} \quad \textrm{and} \quad m=p_1^{b_1} \dots p_N^{b_N}
\end{equation}
where $p_i, 1 \le i \le N$ are all primes and $a_i, b_i \ge 0, 1 \le i \le N$. Moreover, if for some $i \in \{1, \dots, N\}$, $a_i>0$, then $b_i>0$.

We now introduce the concept of \emph{good pair} and \emph{semi-good pair}. 

\begin{defn} \label{goodpair}
Let $m$ and $n$ be defined as in \eqref{20201201eq01}. We say $(m, n)$ is a \emph{semi-good pair} if
\begin{enumerate}
    \item [(a).] $m>n$;
    \item [(b).] $b_i>a_i, \ 1 \le i \le N$.
\end{enumerate}
Moreover, we say $(m, n)$ is a \emph{good pair} if the second condition above is replaced by the following:
\begin{enumerate}
    \item [(c).] $b_i \ge a_i$ for all $1 \le i \le N$, and there exists some $i \in \{1, \dots, N\}$, such that $a_i=b_i>0$.
\end{enumerate}
\end{defn}

\begin{exa}
$(m, n)=(108, 6)$ is a semi-good pair and $(m, n)=(108, 36)$ is a good pair.
\end{exa}

We have the following easy but important observation. 

\begin{lem} \label{20201201lem01}
A good pair is unsolvable. 
\end{lem}

\begin{proof}
Without loss of generality, we assume $a_1=b_1>0$ in condition (c) above. Then for each $\ell \ge 1$, \eqref{20201201eq01} is equivalent to
$$
\frac{k p_2^{a_2 \ell} \dots p_N^{a_N \ell}}{p_2^{b_2(\ell-1)} \dots p_N^{b_N(\ell-1)}}=\frac{1}{p_1^{a_1}},
$$
which clearly has no solutions. 
\end{proof}

\begin{rem}
Lemma \ref{20201201lem01} has a certain geometric interpretation: indeed, let $(m, n)$ be a good pair, then Lemma \ref{20201201lem01} asserts that the point $\frac{1}{n^{\ell}}$ is an endpoint of a $m$-adic interval with sidelength $m^{-\ell}$, since condition (c) implies $\frac{1}{n^\ell}=\frac{k'}{m^\ell}$ for some $k' \in \N$, while it is not an endpoint of a $m$-adic interval with sidelength $m^{-\ell+1}$ since $\frac{1}{n^\ell} \neq \frac{k}{m^{\ell-1}}$ for all $k \in \N$. This exactly suggests us how to find a pair of $m$-adic siblings with the ratio \eqref{20201121eq11} between them ``blowing up".  
\end{rem}

However, Lemma \ref{20201201lem01} in general is \emph{not} true for a semi-good pair. For example, if $(m, n)=(108, 6)$, then for $\ell>2$, there always holds
$$
\frac{2^{\ell-2} \cdot 3^{2\ell-3}}{108^{\ell-1}}=\frac{1}{6^\ell}. 
$$
Nevertheless, we observe that it is indeed ``easy" to modify a semi-good pair into a good pair. More precisely, 

\begin{lem} \label{20201202cor01}
Let $m$ and $n$ a semi-good pair. Let further,
$$
\frac{a}{b}=\max_{1 \le i \le N} \frac{a_i}{b_i}.
$$
Then $(m^a, n^b)$ is a good pair. 
\end{lem}

\begin{proof}
The proof of this lemma is obvious. 
\end{proof}

\subsection{Proof of the main result: the non-far number case} \label{20201207sec01}

In the second half of this section, we complete the proof of Theorem \ref{20201121thm01} under the assumption $\frac{1}{n}$ is not $m$-far. Recall from \eqref{20201201eq01} that this means if $p$ is a prime factor of $n$, then $p$ also divides $m$. We make several reductions.

\medskip

\textit{Step I: Make $(m, n)$ into a semi-good pair.} This is simple. Indeed, we can just take some $\mathfrak{a}$ sufficiently large, such that
\begin{enumerate}
    \item [1.] $m^{\mathfrak{a}}>n$;
    \medskip
    
    \item [2.] $b_i \mathfrak{a}>a_i$ for all $1 \le i \le N$.
\end{enumerate}
This is possible due to \eqref{20201201eq01}. To this end, it suffices to replace $m$ by $m^{\mathfrak{a}}$ and $n$ unchanged; 

\medskip

\textit{Step II: Make $(m, n)$ into a good pair.} This is guaranteed by Lemma \ref{20201202cor01}. 

\bigskip

We make a remark that Theorem \ref{20201121thm01} is indeed invariant under the operation $(m, n) \to \left(m^a, n^b \right)$ for any integers $a, b \ge 1$. This is indeed due to the basic fact that $\calD_m=\calD_{m^a}$ for any $a \ge 1$ (similarly, $\calD_n=\calD_{n^b}$ for any $b \ge 1$). Therefore,  it suffices to prove Theorem \ref{20201121thm01} under the assumption where $(m, n)$ is a good pair.

For $\ell$ sufficiently large, we consider the point
$$
\mathcal{P}_\ell:=\left\lfloor \frac{\ell \log m}{\log n} \right\rfloor-1+\frac{1}{n^\ell}, 
$$
together with two $m$-adic intervals associated to it: 
$$
K_\ell:=\left[ \mathcal{P}_\ell-\frac{1}{m^{\ell}}, \mathcal{P}_{\ell} \right)=\left[  \left\lfloor \frac{\ell \log m}{\log n} \right\rfloor-1+\frac{1}{n^\ell}-\frac{1}{m^\ell}, \left\lfloor \frac{\ell \log m}{\log n} \right\rfloor-1+\frac{1}{n^\ell}       \right)
$$
and
$$
L_\ell:=\left[ \mathcal{P}_\ell, \mathcal{P}_{\ell}+\frac{1}{m^\ell} \right)=\left[ \left\lfloor \frac{\ell \log m}{\log n} \right\rfloor-1+\frac{1}{n^\ell}, \left\lfloor \frac{\ell \log m}{\log n} \right\rfloor-1+\frac{1}{n^\ell}+\frac{1}{m^\ell} \right). 
$$

\medskip

We collect several basic facts about $\mathcal{P_\ell}, K_\ell$ and $L_\ell$:
\begin{enumerate}
    \item [(1).] $\mathcal{P}_\ell=r(K_\ell)=l(L_\ell)$, where recall that for any interval $I$, $r(I)$ is the right endpoint and $l(I)$ is the left endpoint of $I$, respectively;
    
    \medskip
    
    \item [(2).] Both $K_\ell$ and $L_\ell$ are $m$-adic intervals with sidelength $m^{-\ell}$. Indeed, by the definition of good pair $n \mid m$, which suggests $\frac{1}{n^\ell}=\frac{(m/n)^\ell}{m^\ell}$, which gives the desired assertion; 
    
    \medskip
    
    \item [(3).] $K_\ell$ and $L_\ell$ are $m$-adic siblings, that is, there exists some $m$-adic interval with sidelength $m^{-\ell+1}$, such that it contains both $K_\ell$ and $L_\ell$. Indeed, by Lemma \ref{20201201lem01}, for each $\ell \ge 1$, there does not exists a $k \in \N$, such that
    $$
    \left\lfloor \frac{\ell \log m}{\log n} \right\rfloor-1+\frac{1}{n^\ell}=\frac{k}{m^{\ell-1}}.
    $$
    Note that the solubility of the above equation is equivalent to the solubility of the equation \eqref{20201129eq01} since $\left\lfloor \frac{\ell \log m}{\log n} \right\rfloor$ is a positive integer. This implies that $\mathcal{P}_\ell$ can not be an endpoint for any $m$-adic intervals with sidelength $m^{-\ell+1}$, which implies $K_\ell$ and $L_\ell$ are $m$-adic siblings with a common $m$-adic parent $R$ with sidelength $m^{-\ell+1}$.
\end{enumerate}

Finally, we show that the ratio
$$
\frac{\mu\left(K_\ell \right)}{\mu\left(L_\ell \right)}
$$
diverges when $\ell$ tends to infinity, which will then imply Theorem \ref{20201121thm01} with $\frac{1}{n}$ being assumed not $m$-far. Note that on the interval 
$$
\left[ \left\lfloor \frac{\ell \log m}{\log n} \right\rfloor-1, \left\lfloor \frac{\ell \log m}{\log n} \right\rfloor \right),
$$
the construction procedure presented in Section \ref{20201202sub01} repeats $\left\lfloor \frac{\ell \log m}{\log n} \right\rfloor$ times. In particular, this means that near $\mathcal{P}_\ell$, we can find two $n$-adic intervals, which are 
$$
\widetilde{K}_\ell:=\left[ \mathcal{P}_\ell-\frac{1}{n^{\left\lfloor \frac{\ell \log m}{\log n} \right\rfloor}}, \mathcal{P}_{\ell} \right)
$$
and
$$
\widetilde{L}_\ell:=\left[ \mathcal{P}_\ell, \mathcal{P}_{\ell}+\frac{1}{n^{\left\lfloor \frac{\ell \log m}{\log n} \right\rfloor}} \right)
$$
with the associated weights $a^{\ell}b^{\left\lfloor \frac{\ell \log m}{\log n} \right\rfloor-\ell}$ and $a^\ell$, respectively. Moreover, it is also easy to see that
$$
K_\ell \subseteq \widetilde{K}_{\ell} \quad \textrm{and} \quad L_\ell \subseteq \widetilde{L}_{\ell}. 
$$
This is clear from the fact that 
$$
n^{\left\lfloor \frac{\ell \log m}{\log n} \right\rfloor} \le m^\ell.
$$
Therefore, 
$$
\frac{\mu(K_\ell)}{\mu(L_\ell)}=\frac{a^{\ell}b^{\left\lfloor \frac{\ell \log m}{\log n} \right\rfloor-\ell} \cdot \frac{1}{m^\ell}}{a^{\ell} \cdot \frac{1}{m^\ell}}=b^{\left\lfloor \frac{\ell \log m}{\log n} \right\rfloor-\ell}.
$$
It is clear that the last term diverges when $\ell$ converges to infinity since $m>n$ and 
$$
\left\lfloor \frac{\ell \log m}{\log n} \right\rfloor-\ell \ge \left(\frac{\log m}{\log n}-1 \right) \ell-1. 
$$
(see, Figure \ref{20201121Fig05}). The proof is complete. 

\medskip

\begin{figure}[ht]
\begin{tikzpicture}[scale=10]
\draw (-.15,0) -- (1.1,0); 
\fill (0,0) circle [radius=.2pt];
\fill (1,0) circle [radius=.2pt];
\draw [line width=0.8mm, green ] (.3, 0) -- (.5, 0);  
\draw [line width=0.8mm, cyan] (.5, 0) --(.7, 0); 
\fill (0.3, 0) circle [radius=.2pt];
\fill (0.7, 0) circle [radius=.2pt];
\fill (0.3, -.01) node [below] {{\footnotesize $\mathcal{P}_\ell-\frac{1}{m^\ell}$}};
\fill (0.7, -.01) node [below] {{\footnotesize $\mathcal{P}_\ell+\frac{1}{m^\ell}$}};
\draw (0.4, .01) node [above] {{\footnotesize {\color{red} $K_\ell$}}}; 
\draw (0.6, .01) node [above] {{\footnotesize {\color{red} $L_\ell$}}}; 
\fill (.5, 0) circle [radius=.2pt]; 
\fill (-.05, -.01) node [below] {{\footnotesize $\mathcal{P}_\ell-\frac{1}{n^{\left\lfloor \frac{\ell \log m}{\log n} \right\rfloor}}$}};
\fill (0.5, -.01) node [below] {{\footnotesize $\mathcal{P}_\ell$}};
\fill (1, -.01) node [below] {{\footnotesize $\mathcal{P}_\ell+\frac{1}{n^{\left\lfloor \frac{\ell \log m}{\log n} \right\rfloor}}$}};
\draw [decorate,decoration={brace,amplitude=8pt,raise=10pt},yshift=2pt] (0, 0) -- (0.5,  0) node [black,midway,xshift=0cm, yshift=.9cm] {\tiny {{\color{blue} $a^{\ell}b^{\left\lfloor \frac{\ell \log m}{\log n} \right\rfloor-\ell}$}}};
\draw [decorate,decoration={brace,amplitude=8pt,raise=10pt},yshift=-2pt] (0.5, 0) -- (0,  0) node [black,midway,xshift=0cm, yshift=-.9cm] {\footnotesize {{\color{red} $\widetilde{K}_\ell$}}};
\draw [decorate,decoration={brace,amplitude=8pt,raise=10pt},yshift=-2pt] (1, 0) -- (0.5,  0) node [black,midway,xshift=0cm, yshift=-.9cm] {\footnotesize {{\color{red} $\widetilde{L}_\ell$}}};
\draw [decorate,decoration={brace,amplitude=8pt,raise=10pt},yshift=2pt] (0.5, 0) -- (1,  0) node [black,midway,xshift=0cm, yshift=.9cm] {\tiny {{\color{blue} $a^{\ell}$}}};
\end{tikzpicture}
\caption{The non-far number case: $K_\ell, L_\ell, \widetilde{K}_\ell$ and $\widetilde{L}_\ell$.}
\label{20201121Fig05}
\end{figure}
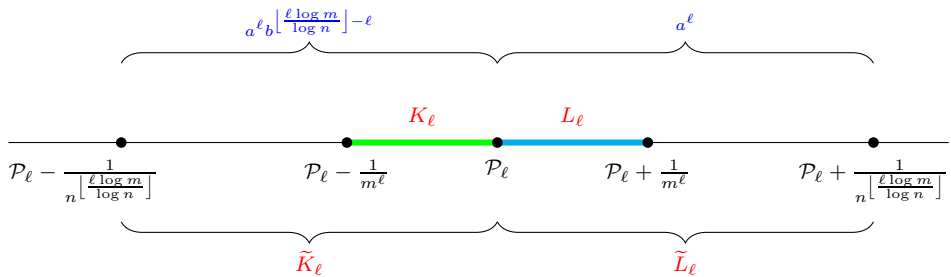

\medskip

As a by product of our main argument to prove Theorem \ref{mainthm}, we have the following a number theoretic classification related to normal numbers (see \cite{Schmidt}, \cite{Wu1}).  Recall

\begin{thm} [\cite{Schmidt}] \label{Schmidtthm}
$\frac{\log m}{\log n}$ is rational if and only if every number normal in base $m$ is also normal in base $n$. \end{thm}

While our approach suggests the following result which classifies a pair of numbers $(m, n)$ which does \emph{not} satisfy Theorem \ref{Schmidtthm}.

\begin{thm} \label{normalcla}
Let $m \ge n$ be two integers with $\frac{\log m}{\log n}$ being irrational.   Then there are only two possible cases:
\begin{enumerate}
    \item [(1).] $\frac{1}{n}$ is $m$-far; 
    \medskip
    \item [(2).] There exists some positive integers $a,b \ge 1$, such that $\left(m^a, n^b\right)$ is a good pair. 
\end{enumerate}
\end{thm}

This result is of independent interest. First of all, it has nothing to do with the construction of doubling measures. Secondly, we can see that this result has potential further applications in studying normal numbers (see \cite{Schmidt}) and other objects in number theory. Thirdly, it suggests a possible connection between the adjacency of general dyadic grids and the collection of $n$-adic doubling measures.

\end{document}